\RequirePackage{fix-cm}
\documentclass[smallextended]{svjour3}       
\smartqed  
\usepackage{graphicx}
\usepackage[numbers]{natbib}
\usepackage{lipsum}
\usepackage{amsfonts}
\usepackage{graphicx}
\usepackage{epstopdf}
\usepackage{algorithmic}
\usepackage{graphicx}
\usepackage{dcolumn}
\usepackage{bm}
\usepackage{mathrsfs}
\usepackage{graphics} 
\usepackage{epsfig} 
\usepackage{times} 
\usepackage{amsmath} 
\usepackage{amssymb}  

\usepackage{amsfonts}
\usepackage{fancyhdr}
\usepackage{color}
\usepackage{subfigure}

\allowdisplaybreaks[1] 

\begin{document}

\title{ISS with Respect to Boundary and In-domain Disturbances for a Coupled Beam-String System
}


\author{Jun Zheng        \and
        Hugo Lhachemi   \and
        \\
        Guchuan Zhu     \and
        David Saussi\'{e}
}

\institute{J. Zheng \at
            {}{School of Mathematics, Southwest Jiaotong University, Chengdu, Sichuan, China 611756} \\
              \email{zhengjun2014@aliyun.com}           
           \and
           H. Lhachemi, G Zhu, and D. Saussi\'{e} \at
           Department of Electrical Engineering, Polytechnique Montr\'{e}al, P.O. Box 6079, Station Centre-Ville, Montreal, QC, Canada H3T 1J4\\
           \email{\{hugo.lhachemi,guchuan.zhu,d.saussie\}@polymtl.ca}
}

\date{Received: date / Accepted: date}

\maketitle

\begin{abstract}
This paper addresses the robust stability of a boundary controlled system coupling two partial differential equations (PDEs), namely beam and string equations, in the presence of boundary and in-domain disturbances under the framework of input-to-state stability (ISS) theory. Well-posedness assessment is first carried out to determine the regularity of the disturbances required for guaranteeing the unique existence of the solution to the considered problem. Then, the method of Lyapunov functionals is applied in stability analysis, which results in the establishment of some ISS properties {}{with respect to} disturbances. As the analysis is based on the \emph{a priori} estimates of the solution to the PDEs, it allows avoiding the invocation of unbounded operators while obtaining the ISS gains in their original expression without involving the derivatives of boundary disturbances.
\end{abstract}

\section{Introduction}\label{Sec: Introduction}
This paper addresses the robust stabilization problem of a boundary controlled system described by a pair of coupled partial differential equations (PDEs) in the presence of boundary and in-domain disturbances. The considered system is a model describing the dynamics in bending and twisting displacement, respectively, for a flexible aircraft wing \cite{Lhachemi:2017}. This model is a linear version of the system presented in \cite{bialy2016adaptive}. A very similar model of a flapping wing UAV is studied in \cite{paranjape2013pde} and \cite{he2017control}. The robust stability analysis presented in this work is carried out in the framework of input-to-state stability (ISS), which was first introduced by Sontag (see \cite{Sontag:1989,Sontag:1990}) and has become one of the central concepts in the study of robust stability of control systems.

During the last two decades, a complete theory of ISS for nonlinear finite dimensional systems has been established and has been successfully applied to a very wide range of problems in nonlinear systems analysis and control (see, e.g., \cite{karafyllis2011stability}). In recent years, a considerable effort has been devoted to extending the ISS theory to infinite dimensional systems governed by partial differential equations, including the characterization of ISS and iISS (integral input-to-state stability, which is a variant of ISS \cite{Sontag:1998}) \cite{Dashkovskiy:2013, Jacob:2016, Jayawardhana:2008, Mironchenko:2016, Mironchenko:2014a, Mironchenko:2014b, Mironchenko:2014, Mironchenko:2015, Mironchenko:2016c, Mironchenko:2016b, Mironchenko:2017} and the establishment of ISS properties for different PDE systems {}{\cite{Argomedo:2012,Argomedo:2013, Dashkovskiy:2010, Dashkovskiy:2013b, jacob2016input, Karafyllis:2014, Karafyllis:2016a, Karafyllis:2016}, {}{\cite{karafyllis2017siam, Karafyllis:2018, Karafyllis:2018b}}, \cite{Logemann:2013, Mazenc:2011, Prieur:2012, Tanwani:2016_CDC,Tanwani:2017,Zheng:201702,Zheng:2017,Zheng:2018}}.

In the formulation of PDEs, disturbances can be distributed over the domain and/or appear at isolated points in the domain or on the boundaries. Usually, pointwise disturbances will lead to a formulation involving unbounded operators \cite{jacob2016input, Karafyllis:2016, Karafyllis:2016a, Mironchenko:2014b}, which is considered to be more challenging than the case of distributed disturbances \cite{Karafyllis:2016}. To avoid dealing with unbounded operators, it is proposed in \cite{Argomedo:2012} to transform the boundary disturbance to a distributed one, which allows for the application of the tools established for the latter case, in particular the method of Lyapunov functionals. However, it is pointed out in \cite{Karafyllis:2016, Karafyllis:2016a} that such a method will end up establishing the ISS property {}{with respect to} boundary disturbance and some of its time derivatives, which is not strictly in the original form of ISS formulation. For this reason, the authors of \cite{Karafyllis:2016, Karafyllis:2016a} {}{proposed a finite-difference scheme and eigenfunction expansion method with which the ISS
in $L^2$-norm and in weighted $L^\infty$-norm is derived directly from the estimates of the solution to the considered PDEs associated with a Sturm-Liouville operator.} Although the aforementioned transformation of the disturbance from the boundary to the domain is still used, it is only for the purpose of well-posedness assessment, while the ISS property is expressed solely in terms of disturbances as expected. Nevertheless, the {}{method} employed in \cite{Karafyllis:2016, Karafyllis:2016a} may involve a very heavy computation when dealing with higher-order, coupled PDEs with complex boundary conditions including disturbances, as the one considered in the present work.

{}{A monotonicity-based method has been introduced in \cite{Mironchenko:2017} for studying the ISS of nonlinear parabolic equations with boundary disturbances. It has been shown that with the monotonicity the ISS of the original nonlinear parabolic PDE with constant boundary disturbances is equivalent to the ISS of a closely related nonlinear parabolic PDE with constant distributed disturbances and zero boundary conditions. As an application of this method, the ISS properties in $L^p$-norm ({}{$\forall p>2$}) for some linear parabolic systems have been established.}

{}{It has been shown in \cite{Zheng:201702} and \cite{Zheng:2017} that the classical method of Lyapunov functionals is still effective in obtaining ISS properties w.r.t. boundary disturbances for certain semilinear parabolic PDEs with Dirichlet, and Neumann (or Robin) boundary conditions, respectively. In \cite{Zheng:201702}, the technique of De~Giorgi iteration is used when Lyapunov method is involved in the establishment of ISS for PDEs with Dirichlet boundary disturbances. ISS in $L^2$-norm for Burgers' equations, and ISS in $L^\infty$-norm for some linear PDEs, have been established in \cite{Zheng:201702}. In \cite{Zheng:2017}, some technical inequalities have been developed, which allows dealing directly with the boundary disturbances in proceeding on ISS in $L^2$-norm for certain semilinear PDEs with Neumann (or Robin) boundary conditions via Lyapunov method. In \cite{Tanwani:2016_CDC}, the ISS w.r.t. boundary disturbances in $H^1$-norm has also been established for linear hyperbolic {PDEs} using Lyapunov method.}

It should be noticed that it is shown in \cite{jacob2016input} that for a class of linear PDEs with boundary disturbances, {}{under the assumption that the semigroup is exponential stable, ISS is equivalent to iISS, with the aid of admissibility.} Nevertheless, this is a quite strong condition and there may be difficulties to apply this assertion to systems for which the associated operators are not \textit{a priori} dissipative, as dissipativity is a non-trivial property depending closely on, among other factors, the boundary conditions and the regularity of the disturbances.

The method adopted in the present work is also the application of Lyapunov theory in the establishment of the ISS and iISS properties of the considered system with respect to boundary and in-domain disturbances. However, greatly inspired by the methodology proposed in \cite{Karafyllis:2016, Karafyllis:2016a, Zheng:2017}, stability analysis is based on the \textit{a priori} estimates of the solution to the original PDEs, which allows avoiding the invocation of unbounded operators while obtaining the ISS and iISS properties expressed only in terms of the disturbances. The development of the solution consists in two steps. In the first step, we perform a well-posedness analysis to determine the regularity of the disturbances required for ensuring the existence of the solutions to the PDEs. Similar to \cite{Argomedo:2012,Karafyllis:2016, Karafyllis:2016a}, the technique of lifting is used in well-posedness analysis to avoid involving unbounded operators. In the second step, the ISS and iISS properties are established via the estimates of the solution to the original system. {}{Instead of dealing with certain energy functional directly, the Lyapunov functional candidate for the system is actually derived from the regularity analysis of the solutions.} {In general, a Lyapunov functional candidate may be chosen according to the norms of the solution and their derivatives arising in the computation of \textit{a priori} estimates of the solutions.}

Note that the result presented in this work demonstrates that {the appearance of the derivatives of boundary disturbances in ISS or iISS gains
is not necessarily inherent to the Lyapunov method and may be avoided for certain settings}. Therefore, we can expect that the well-established method of Lyapunov functionals can be applied to the establishment of ISS properties {}{with respect to} boundary disturbances for a wide range of PDEs. This constitutes the main contribution of the present work.

In the remainder of the paper, {}{
Section~\ref{Sec: Problem formulation} introduces the dynamic model of the coupled beam-string system and presents the well-posedness assessment. Section~\ref{Sec: Stability} is devoted to the analysis of ISS and iISS properties of the considered system. Numerical simulation results for the considered system are presented in Section~\ref{Sec: Simulation}, followed by concluding remarks given in Section~\ref{Sec: Conclusion}.}

\section{{}{Problem formulation and Well-posedness Analysis}}\label{Sec: Problem formulation}

\subsection{{}{Notation}}\label{Sec: notation}
{}{Let $\mathbb{R}=(-\infty,+\infty), \mathbb{R}_{+}=(0,+\infty)$, and $\mathbb{R}_{\geq 0} =  \{0\}\cup \mathbb{R}_{+}$.} {}{We define some function spaces for functions with
one variable. For $a,b$$\in $$[-\infty,$$+\infty]$ and $p\in [1,+\infty)$, $L^p(a,b)$ is the space of all measurable functions $f$ whose absolute value raised to the $p^{\text{th}}$-power has a finite integral. The norm $\|\cdot\|$ on $L^p(a,b)$ is defined by $ \|f\|_{L^p(a,b)}=\left(\int_{a}^b|f(x)|^p\text{d}x\right)^{\frac{1}{p}}$. $L^{\infty}(a,b)$ is the space all measurable functions $f$ whose absolute value is essential bounded. The norm $\|\cdot\|$ on $L^\infty(a,b)$ is defined by  $\|f\|_{L^{\infty}(a,b)}=\text{ess}\sup\limits_{a< x<b}|f(x)|.$
For a positive integer $m$, $H^m(a,b)=H^m((a,b);\mathbb{R})=\{f:(a,b)\rightarrow\mathbb{R}
|\ f\in L^2(a,b)$ {with each $s$-th order} {weak derivative} {$D^s f\in L^2(a,b),\ s=1,2,\ldots,m\}$}.  For a nonnegative integer $m$, $C^m(\mathbb{R}_{\geq 0})=C^m(\mathbb{R}_{\geq 0};\mathbb{R})=\{f:\mathbb{R}_{\geq 0}\rightarrow \mathbb{R}|\frac{\text{d}^sf}{\text{d}x^s} (s=0,1,2,\ldots,m)$ exist and are continuous on $\mathbb{R}_{\geq 0}\}$.
}

{}{We define some function spaces for functions with two variables.
For $t\in \mathbb{R}_{\geq 0},l\in \mathbb{R}_{\geq 0}$ and $1\leq p<+\infty$, the space $L^{\infty}(0,t;L^p(0,l))$ consists of all strongly measurable functions $f:[0,t]\rightarrow L^p(0,l)$ with the norm
\begin{align*}
 \|f\|_{L^{\infty}(0,t;L^p(0,l))}= \text{ess}\sup\limits_{0< s< t}\|f(\cdot,s)\|_{L^{p}(0,l)} <+\infty.
\end{align*}
The space $L^{\infty}(0,t;L^\infty(0,l))$ consists of all strongly measurable functions $f:[0,t]\rightarrow L^\infty(0,l)$ with the norm
\begin{align*}
\|f\|_{L^{\infty}(0,t;L^\infty(0,l))} = \text{ess}\sup\limits_{0< s< t}\|f(\cdot,s)\|_{L^{\infty}(0,l)} <+\infty.
\end{align*}
For a nonnegative integer $m$ and a vector space $H$, $C^m(\mathbb{R}_{\geq 0};H)=\{f:\mathbb{R}_{\geq 0}\rightarrow H|\frac{\partial ^s f}{\partial t^s} (\cdot,t) \in H$, and $\frac{\partial ^s f}{\partial t^s} (\cdot,t)$ is continuous on $\mathbb{R}_{\geq 0},\ s=0,1,2,\ldots,m\}$.}

Some well-known function classes commonly used in Lyapunov-based stability analysis are specified below:
\newline $\mathcal {K}=\{\gamma : \mathbb{R}_{\geq 0} \rightarrow \mathbb{R}_{\geq 0}|\ \gamma(0)=0,\gamma$ is continuous, strictly increasing$\}$;
\newline $ \mathcal {K}_{\infty}=\{\theta \in \mathcal {K}|\ \lim\limits_{s\rightarrow\infty}\theta(s)=\infty\}$;
\newline $ \mathcal {L}=\{\gamma : \mathbb{R}_{\geq 0}\rightarrow \mathbb{R}_{\geq 0}|\ \gamma$ is continuous, strictly decreasing, $\lim\limits_{s\rightarrow\infty}\gamma(s)=0\}$;
\newline  $ \mathcal {K}\mathcal {L}=\{\beta : \mathbb{R}_{\geq 0}\times \mathbb{R}_{\geq 0}\rightarrow \mathbb{R}_{\geq 0}|\ \beta(\cdot,t)\in \mathcal {K}, \forall t \in \mathbb{R}_{\geq 0}$, and $\beta(s,\cdot)\in \mathcal {L}, \forall s \in {}{ \mathbb{R}_{+}}\}$.

\subsection{{}{System setting}}
Let $l\in \mathbb{R}_{\geq 0}$ be the length of the wing. Denote by $w(y,t): [0,l]\times \mathbb{R}_{\geq 0} \rightarrow \mathbb{R}$ and $\phi(y,t): [0,l]\times \mathbb{R}_{\geq 0} \rightarrow \mathbb{R}$ the bending and twisting displacements, respectively, at the location $y \in [0, l]$ along the wing span and at time $t \geq 0$. In the present work, we consider the dynamics of a flexible aircraft wing expressed by the following initial-boundary value problem (IBVP) representing a coupled beam-string system with boundary control \cite{Lhachemi:2017}:
\begin{subequations}\label{system 1426}
\begin{align}
  &w_{tt}+ (a_1 w_{yy} + b_1 w_{tyy})_{yy} =c_1 \phi + p_1 \phi_t + q_1 w_t + d_{1},\label{1426a}\\
  &\phi_{tt}- (a_2 \phi_{y} + b_2\phi_{ty})_{y} = c_2 \phi + p_2 \phi_t + q_2 w_t + d_{2},\label{1426b}\\
  \begin{split}\label{1426c}
  &w(0,t)=w_y(0,t)=\phi(0,t)=0, \\
  &(a_1 w_{yy} + b_1 w_{tyy})_{y}(l,t) = d_{3}(t),
  \; (a_2 \phi_y + b_2 \phi_{ty})(l,t) = d_{4}(t),
  \end{split}\\
  & w(y,0) = w^0, w_t(y,0) = w^0_1, \phi(y,0) = \phi^0, \phi_t(y,0) = \phi_1^0, \label{1426d}
\end{align}
\end{subequations}
where \eqref{1426a} and \eqref{1426b} are defined in $(0,l)\times\mathbb{R}_{\geq 0}$, $a_i>0$, $b_i>0$, $c_i\geq 0$ ($i=1,2$), {}{$p_1\geq 0$, $p_2\leq 0$, $q_1\leq 0$ and $q_2\geq 0$} are constants depending on structural and aerodynamic parameters, $w^0,w^0_1\in H^2(0,l)$, $\phi^0,\phi^0_1\in H^1(0,l)$, $d_1,d_2 \in \ C^1(\mathbb{R}_{\geq 0};L^2(0,l))$, and $d_3,d_4 \in \ C^2(\mathbb{R}_{\geq 0};\mathbb{R})$. Functions $d_1(y,t)$ and $d_2(y,t)$ represent disturbances distributed over the domain, while functions $d_{3}(t)$ and $d_{4}(t)$ represent disturbances at the boundary $y=l$. In general, $d_1$ and $d_2$ can represent modeling errors and aerodynamic load perturbations, and $d_3$ and $d_4$ can represent actuation and sensing errors.
\begin{remark}
{}{ \eqref{system 1426} is a model of flexible aircraft wing with Kelvin-Voigt damping, in which the constants $\frac{b_1}{a_1}$ in \eqref{1426a} and $\frac{b_2}{a_2}$ in \eqref{1426b} represent the coefficients of bending Kelvin-Voigt damping and torsional Kelvin-Voigt damping respectively (see \cite{Lhachemi:2017} for instance).}
\end{remark}

\subsection{{}{Well-posedness analysis}}\label{Sec: Well-posedness analysis}

In this subsection, we prove the well-posedness of System \eqref{system 1426}. To this end, consider the Hilbert space
\begin{align*}
\mathcal{H} :=   \Big\{ & (f,g,h,z) \in H^2(0,l) \times L^2(0,l) \times H^1(0,l) \times L^2(0,l) : \\
& f(0)=f_y(0)=h(0)=0, f,f_y,h\in\mathrm{AC}[0,l]\Big\} ,
\end{align*}
endowed with the inner product
\begin{align*}
 & \langle(f_1,g_1,h_1,z_1),(f_2,g_2,h_2,z_2)\rangle_{\mathcal {H}} = \int_{0}^{l}(a_1 f_{1yy} f_{2yy} + g_1 g_2 + a_2 h_{1y} h_{2y} + z_1 z_2)\text{d}y.
\end{align*}
Introducing the state vector $X = (f,g,h,z)$, the norm ${\|\cdot\|_{\mathcal{H}}}$ on $\mathcal{H}$ induced by the inner product can be expressed as:
\begin{align*}
\|X\|^2_{\mathcal {H}} = & \|\sqrt{a_1} f_{yy}\|^2_{L^{2}(0,l)} +\|g\|^2_{L^{2}(0,l)} +\|\sqrt{a_2} h_{y}\|^2_{L^{2}(0,l)} + \|z\|^2_{L^{2}(0,l)}.
\end{align*}

In order to reformulate System \eqref{system 1426} in an abstract form evolving in the space $\mathcal{H}$, we define the following operators. First, we introduce the unbounded operator $\mathcal{A}_{1,d} : D(\mathcal{A}_{1,d}) \subset \mathcal{H} \rightarrow \mathcal{H}$ defined by
\begin{equation}
\mathcal{A}_{1,d} X := \left(g,-(a_1 f_{yy} + b_1 g_{yy})_{yy},z,(a_2 h_y + b_2 z_y)_y\right)
\end{equation}
on the following domain:
\begin{align*}
D(\mathcal{A}_{1,d}) := \Big\{ & (f,g,h,z)\in\mathcal{H} :
 g\in H^2(0,l) ,\; z\in H^1(0,l),\\
& (a_1 f_{yy} + b_1 g_{yy}) \in H^2(0,l) ,
 (a_2 h_y + b_2 z_y) \in H^1(0,l), \\
& f(0)=f_y(0)=0 , \; g(0)=g_y(0)=0 ,\\
& h(0)=0 ,\; z(0)=0 , (a_1 f_{yy} + b_1 g_{yy})(l) = 0 , \\
& f,f_y,g,g_y,h,z,(a_1 f_{yy} + b_1 g_{yy}) \in \mathrm{AC}[0,l] , \\
& (a_1 f_{yy} + b_1 g_{yy})_y , (a_2 h_y + b_2 z_y) \in \mathrm{AC}[0,l] \Big\},
\end{align*}
where $\mathrm{AC}[0,l]$ denotes the set of all absolutely continuous functions on $[0,l]$.
The contribution of other terms are embedded into the bounded operator $\mathcal{A}_{2}\in\mathcal{L}(\mathcal{H})$ defined as
\begin{equation}
\mathcal{A}_{2}X := (0,c_1 h + p_1 z + q_1 g,0,c_2 h + p_2 z + q_2 g),
\end{equation}
with domain $D(\mathcal{A}_{2}) = \mathcal{H}$ (the bounded property is a direct consequence of the Poincar{\'e}'s inequality). Finally, we consider the boundary operator $\mathcal{B} : D(\mathcal{B}) = D(\mathcal{A}_{1,d}) \rightarrow \mathcal{H}$ defined as
\begin{equation}
\mathcal{B}X := ((a_1 f_{yy} + b_1 g_{yy})_y(l),(a_2 h_y + b_2 z_{y})(l)).
\end{equation}
Thus, System \eqref{system 1426} can be represented in the following abstract system:
\begin{equation}\label{eq: pb abstract form 1}
\left\{\begin{split}
\dot{X} & = \left[ \mathcal{A}_{1,d} + \mathcal{A}_2 \right] X + (0,d_1,0,d_2) \\
\mathcal{B}X & = U \\
X_0 & \in D(\mathcal{A}_{1,d}) ,\; \mathrm{s.t.} \; \mathcal{B}X_0 = U(0)
\end{split}\right.
\end{equation}
where $U \triangleq (d_3,d_4)$.

In order to assess the well-posedness of (\ref{eq: pb abstract form 1}), we introduce the unbounded disturbance free operator $\mathcal{A}_1 = D(\mathcal{A}_1) \subset \mathcal{H} \rightarrow \mathcal{H}$ defined on the domain $D(\mathcal{A}_1) = D(\mathcal{A}_{1,d}) \cap \mathrm{ker}(\mathcal{B})$ by $\mathcal{A}_1 = \left.\mathcal{A}_{1,d}\right|_{D(\mathcal{A}_1)}$. We also consider the lifting operator {}{$\mathcal{T} \in \mathcal{L} (\mathbb{R}^2,\mathcal{H})$} defined by
\begin{equation}
\mathcal{T}(d_3,d_4) := \left( y \rightarrow -\dfrac{d_3}{6 a_1} y^2(3l-y) , 0 , y \rightarrow \dfrac{d_4}{a_2} y , 0 \right) .
\end{equation}
with $\left\Vert \mathcal{T} \right\Vert = \sqrt{l \times \max(1/a_2,l^2/(3 a_1))}$ when $\mathbb{R}^2$ is endowed with the usual $l^2$-norm. A direct computation shows that $R(\mathcal{T})\subset D(\mathcal{A}_{1,d})$, $\mathcal{A}_{1,d}\mathcal{T} = 0_{\mathcal{L}(\mathbb{R}^2,\mathcal{H})}$ and $\mathcal{B}\mathcal{T} = I_{\mathbb{R}^2}$, {}{where $R(\mathcal{T})$ is the range of the operator $\mathcal{T}$.} Thus, we can define a system in the following abstract form:
\begin{equation}\label{eq: pb abstract form 2}
\left\{\begin{split}
\dot{V} & = \left[ \mathcal{A}_1 + \mathcal{A}_2 \right] V + \mathcal{A}_2 \mathcal{T} U - \mathcal{T} \dot{U} + (0,d_1,0,d_2) \\
V_0 & \in D(\mathcal{A}_1)
\end{split}\right.
\end{equation}

{}{By \cite[Th 3.3.3]{Curtain:1995}, we have the following relationship between the solutions of abstract systems (\ref{eq: pb abstract form 1}) and (\ref{eq: pb abstract form 2}).}
\begin{lemma}\label{lemma: lifting}
Let $X_0\in D(\mathcal{A}_{1,d})$, $d_1,d_2\in\ C^1(\mathbb{R}_{\geq 0};L^2(0,l))$, and $d_3,d_4\in\ C^2(\mathbb{R}_{\geq 0};\mathbb{R})$ such that $\mathcal{B}X_0 = (d_3(0),d_4(0))$. Then $X\in\ C^0(\mathbb{R}_{\geq 0};D(\mathcal{A}_{1,d}))\cap\ C^1(\mathbb{R}_{\geq 0};\mathcal{H})$ with $X(0)=X_0$ is a solution of (\ref{eq: pb abstract form 1}) if and only if $V=X-\mathcal{T}U\in\ C^0(\mathbb{R}_{\geq 0};D(\mathcal{A}_{1}))\cap\ C^1(\mathbb{R}_{\geq 0};\mathcal{H})$ is a solution of (\ref{eq: pb abstract form 2}) for the initial condition $V_0 = X_0 - \mathcal{T}U(0)$.
\end{lemma}

We can now use Lemma~\ref{lemma: lifting} to assess the well-posedness of the original abstract problem (\ref{eq: pb abstract form 1}).

\begin{theorem}
For any $d_1,d_2\in\ C^1(\mathbb{R}_{\geq 0};L^2(0,l))$, and $d_3,d_4\in\ C^2(\mathbb{R}_{\geq 0};\mathbb{R})$, the abstract problem (\ref{eq: pb abstract form 1}) admits a unique solution $X\in\ C^0(\mathbb{R}_{\geq 0};D(\mathcal{A}_{1,d}))\cap\ C^1(\mathbb{R}_{\geq 0};\mathcal{H})$ for any given $X_0\in D(\mathcal{A}_{1,d})$ such that $\mathcal{B}X_0 = (d_3(0),d_4(0))$.
\end{theorem}

\begin{proof}
Let $X_0\in D(\mathcal{A}_{1,d})$ such that $\mathcal{B}X_0 = U(0)$. It is known that $\mathcal{A}_1$ generates a $C^0$-semigroup on $\mathcal{H}$ \cite{Lhachemi:2017}. As {}{$\mathcal{A}_{2}\in\mathcal{L}(\mathcal{H})$}, $\mathcal{A}_1+\mathcal{A}_2$ generates a $C^0$-semigroup on $\mathcal{H}$ {}{(see \cite[Th 3.2.1]{Curtain:1995})}. Furthermore, $\mathcal{A}_2\mathcal{T}U-\mathcal{T}\dot{U}+(0,d_1,0,d_2)\in\ C^1(\mathbb{R}_{\geq 0};\mathcal{H})$ {}{due to $\mathcal{T} \in \mathcal{L} (\mathbb{R}^2,\mathcal{H})$ and $\mathcal{A}_{2}\in\mathcal{L}(\mathcal{H})$.} Then, from \cite[Th 3.1.3]{Curtain:1995}, (\ref{eq: pb abstract form 2}) admits a unique solution $V\in\ C^0(\mathbb{R}_{\geq 0};D(\mathcal{A}_{1}))\cap\ C^1(\mathbb{R}_{\geq 0};\mathcal{H})$ for the initial condition $V(0)=V_0=X_0-\mathcal{T}U(0)$. We deduce then from Lemma~\ref{lemma: lifting} that there exists a unique solution $X\in\ C^0(\mathbb{R}_{\geq 0};D(\mathcal{A}_{1,d}))\cap\ C^1(\mathbb{R}_{\geq 0};\mathcal{H})$ to (\ref{eq: pb abstract form 1}) associated to the initial condition $X(0)=X_0$.
\end{proof}


\section{{}{Stability Assessment}}\label{Sec: Stability}
In this section we establish the stability property of System~\eqref{system 1426}. {}{Let $D(\mathcal{A}_{1,d})$, $\mathcal{H}$ and the norm $\|\cdot\|_{\mathcal{H}}$ be defined as in Section \ref{Sec: Well-posedness analysis}. Let $(w,\phi)$ be the unique solution of System~\eqref{system 1426} satisfying $(w,w_t,\phi,\phi_t)\in \ C^0(\mathbb{R}_{\geq 0};D(\mathcal{A}_{1,d}))\cap\ C^1(\mathbb{R}_{\geq 0};\mathcal{H})$. {}{For simplicity, throughout this section, we express the state variable and its initial value as} $X=(w,w_t,\phi,\phi_t)$ and $X_0=(w^0,w_1^0,\phi^0,\phi_1^0)$.} {}{Define the energy function
\begin{align}\label{+12}
E(t)=\frac{1}{2}\int_{0}^{l}\big(|w_t|^2+a_1|w_{yy}|^2+|\phi_t|^2+a_2|\phi_{y}|^2\big)\text{d}y.
\end{align}
Then $ \|X(\cdot,t)\|_{\mathcal{H}}^2=2E(t)$ for all $t\geq 0$.}
\begin{definition}
System~\eqref{system 1426} is said to be input-to-state stable (ISS) {}{with respect to} disturbances $d_{1},d_{2}\in C^{1}(\mathbb{R}_{\geq 0};{L^{2}(0,l)})$ and ${d_{3},d_4}\in C^2(\mathbb{R}_{\geq 0})\cap L^{\infty}( \mathbb{R}_{\geq 0})$, if there exist functions $ \gamma_1, \gamma_2, \gamma_3, \gamma_4 \in \mathcal {K}$ and $\beta\in \mathcal {K}\mathcal {L}$ such that the solution of System~\eqref{system 1426} satisfies
{}
{\begin{align}\label{Eq: ISS def}
     \|X(\cdot,t)\|_{\mathcal{H}} &\leq \beta( \|X_0\|_\mathcal{H},t)
      +\gamma_1(\|d_{1}\|_{L^{\infty}(0,t;L^2(0,l))}) +\gamma_2(\|d_{2}\|_{L^{\infty}(0,t;L^2(0,l))}) \nonumber\\
     & \ \ \ \ \ \ +\gamma_3(\|d_{3}\|_{L^{\infty}(0,t)})+\gamma_4(\|d_{4}\|_{L^{\infty}(0,t)}),\ \forall t\geq 0.
\end{align}}
Moreover, System~\eqref{system 1426} is said to be exponential input-to-state stable (EISS) {}{with respect to} disturbances $d_1$, $d_2$, $d_3$, and $d_4$ if there exist $\beta'\in \mathcal {K}_{\infty}$ and a constant $\lambda > 0$ such that \eqref{Eq: ISS def} holds with $\beta( \|X_0\|_\mathcal{H},t) = \beta'(\|X_0\|_\mathcal{H})e^{-\lambda t}$.
\end{definition}

\begin{definition}
System~\eqref{system 1426} is said to be integral input-to-state stable (iISS) {}{with respect to} disturbances $d_{1},d_{2}\in C^{1}(\mathbb{R}_{\geq 0};{L^{2}(0,l)})$ and ${d_{3},d_4}\in C^2(\mathbb{R}_{\geq 0})\cap L^{\infty}( \mathbb{R}_{\geq 0})$, if there exist functions $\beta\in \mathcal {K}\mathcal {L},\theta_1,\theta_2,\theta_3,\theta_4\in \mathcal {K}_{\infty} $ and $\gamma_1 ,\gamma_2 ,\gamma_3 ,\gamma_4\in \mathcal {K}$, such that the solution of System~\eqref{system 1426} satisfies
{}{\begin{align}
\begin{split}\label{Eq: iISS def}
      \|X(\cdot,t)\|_{\mathcal{H}}&\leq  {\beta( \|X_0\|_\mathcal{H},t)}
          +\theta_{1}\bigg(\!\!\int_{0}^t\!\!\gamma_1(\|d_{1}(\cdot,s)\|_{L^2(0,l)})\text{d}s\bigg)\\
       &\ \ \ \ \ \  +{\theta_{2}}\bigg(\!\!\int_{0}^t\!\!{\gamma_2}(\|d_{2}(\cdot,s)\|_{L^2(0,l)})\text{d}s\bigg)\\
       &\ \ \ \ \ \  +{\theta_{3}}\bigg(\!\!\int_{0}^t\!\!{\gamma_3}(|d_3(s)|)\text{d}s\bigg) +{\theta_{4}}\bigg(\!\!\int_{0}^t\!\!{\gamma_4}(|d_4(s)|)\text{d}s\bigg),\ \forall t\geq 0.
\end{split}
\end{align}}
Moreover, System~\eqref{system 1426} is said to be exponential integral input-to-state stable (EiISS) {}{with respect to} disturbances $d_1$, $d_2$, $d_3$, and $d_4$ if there exist $\beta'\in \mathcal {K}_{\infty}$ and a constant $\lambda > 0$ such that \eqref{Eq: iISS def} holds with $\beta( \|X_0\|_\mathcal{H},t) = \beta'(\|X_0\|_\mathcal{H})e^{-\lambda t}$.
\end{definition}

In order to obtain the stability of the solutions, we make the following {}{assumptions}:
\begin{subequations}\label{assumption}
\begin{align}
  &l^2\sqrt{2l}\|d_3\|_{L^\infty( \mathbb{R}_{\geq 0})}<2a_1,\label{8a}\\
  &\sqrt{2l}(1+l\sqrt{l})(1+K_m)(1+c_1+c_2{}{-p_{2}+{}{q_{2}}} +\|d_4\|_{L^\infty( \mathbb{R}_{\geq 0})})<a_2,\label{8b}\\
  &l^2\sqrt{2l}(1+l^3)(c_1+p_{1}{}{-q_1}+q_{2} +\|d_3\|_{L^\infty( \mathbb{R}_{\geq 0})})<2b_1,\label{8c}\\
  &\sqrt{2l}(1+l^3)(1+p_{1}+c_2{}{-p_{2}}+q_{2} +\|d_4\|_{L^\infty( \mathbb{R}_{\geq 0})})<b_2,\label{8d}
\end{align}
\end{subequations}
where $K_m=\max\Big\{\frac{1}{\sqrt{a_1}},$ $\frac{1}{\sqrt{a_2}}, \frac{l^2}{2\sqrt{a_2}},\frac{l^4}{4\sqrt{a_1}}\Big\}$.

{}{For notational simplicity, we denote hereafter $\|\cdot\|_{L^{2}(0,l)}$ by $\|\cdot\|$.}
\begin{theorem}\label{Th: EISpS}
{Assume that}
\begin{itemize}
\item[(i)] $d_{1},d_{2}\in C^1(\mathbb{R}_{\geq 0};L^2(0,l))$;
\item[(ii)] ${d_{3},d_4}\in C^2(\mathbb{R}_{\geq 0})\cap L^{\infty}( \mathbb{R}_{\geq 0})$;
\item[(iii)] all conditions in \eqref{assumption} are satisfied.
\end{itemize}
Then System~\eqref{system 1426} is EISS and EiISS, having the following estimates:
{}{\begin{align}
  \|X(\cdot,t)\|_{\mathcal{H}}
\leq & C  e^{ -\frac{\mu_m}{4}t}\|X_0\|_{\mathcal{H}} + C\Big(\|d_{1}\|_{L^{\infty}(0,t;L^2(0,l))} +\|d_{2}\|_{L^{\infty}(0,t;L^2(0,l))}\notag\\
     & +\|d_3\|^{\frac{1}{2}}_{L^{\infty}(0,t)}+\|d_4\|^{\frac{1}{2}}_{L^{\infty}(0,t)},\label{main result 1}
\end{align}
and
\begin{align}
\|X(\cdot,t)\|_{\mathcal{H}}\leq & C e^{ -\frac{\mu_m}{4}t}\|X_0\|_{\mathcal{H}} +C\bigg(\int_{0}^t\|d_{1}(\cdot,s)\|^2\text{d}s\bigg)^{\frac{1}{2}}+C\bigg(\int_{0}^t\|d_2(\cdot,s)\|^2\text{d}s\bigg)^{\frac{1}{2}}\notag\\
         &+ C\bigg(\int_{0}^t\big(|d_{3}(s)|\text{d}s\bigg)^{\frac{1}{2}}+C\bigg(\int_{0}^t|d_{4}(s)|\text{d}s\bigg)^{\frac{1}{2}} .\label{main result 2}
   \end{align}
where $C>0$ and $\mu_m>0$ are some constants independent of $t$.}

\end{theorem}

\begin{proof}{}{We introduce first the following notations:
\begin{align*}
  &f_1(\phi,\phi_t,w_t,d_{1})=c_1\phi+p_{1}\phi_t+q_{1}w_t+d_{1},\\
  &f_2(\phi,\phi_t,w_t,d_{2})=c_2\phi+p_{2}\phi_t+q_{2}w_t+d_{2}.
\end{align*}}
 In order to find an appropriate Lyapunov functional candidate, multiplying \eqref{1426a} by $w_t$ and considering the fact that $w\in C^{1}(\mathbb{R}_{\geq 0}; H^2(0,l))\cap C^{2}(\mathbb{R}_{\geq 0}; L^2(0,l))$ with $ (a_1w_{yy}+b_1w_{tyy})(\cdot,t)\in H^{2}(0,1)$, we get
\begin{align*}
&\int_{0}^{l}f_1(\phi,\phi_t,w_t,d_{1})w_t\text{d}y
     =\int_{0}^{l}(w_{tt}+(a_1w_{yy}+b_1w_{tyy})_{yy})w_t\text{d}y\notag\\
 =&\int_{0}^{l}w_{tt}w_t\text{d}y+a_{1}\int_{0}^{l}w_{yy}w_{tyy}\text{d}y
      +b_1\int_{0}^{l}w_{tyy}^2\text{d}y+d_{3}(t)w_t(l,t)\notag\\
 =&\frac{1}{2}\frac{d}{dt}\big(\|w_t\|^2+a_1\|w_{yy}\|^2\big)+b_1\|w_{tyy}\|^2+d_{3}(t)w_t(l,t),
_{}\end{align*}
which gives
\begin{align}\label{+8}
\frac{1}{2}\frac{d}{dt}\big(\|w_t\|^2 &+a_1\|w_{yy}\|^2\big) \notag\\
  =& - b_1\|w_{tyy}\|^2-d_{3}(t)w_t(l,t) +\int_{0}^{l}f_1(\phi,\phi_t,w_t,d_{1})w_t\text{d}y.
\end{align}
Multiplying \eqref{1426a} by $\phi_t$ and since $\phi \in C^{1}(\mathbb{R}_{\geq 0}; H^1(0,l))\cap C^{2}(\mathbb{R}_{\geq 0}; L^2(0,l))$ with $(a_2\phi_{y}+b_2\phi_{ty})(\cdot,t)\in  H^1(0,l)$, we get
\begin{align*}
&\int_{0}^{l}f_2(\phi,\phi_t,w_t,d_{2})\phi_t\text{d}y =\int_{0}^{l}(\phi_{tt}
- (a_2\phi_{y}+b_2\phi_{ty})_{y})\phi_t\text{d}y\notag\\
 =&\frac{1}{2}\frac{d}{dt}\big(\|\phi_t\|^2+a_2\|\phi_{y}\|^2\big) +b_2\|\phi_{ty}\|^2-d_{4}(t)\phi_t(l,t),
\end{align*}
which gives
\begin{align}\label{+9}
&\frac{1}{2}\frac{d}{dt}\big(\|\phi_t\|^2+a_2\|\phi_{y}\|^2\big)
 = -b_2\|\phi_{ty}\|^2+d_{4}(t)\phi_t(l,t) +\int_{0}^{l}\!\!f_2(\phi,\phi_t,w_t,d_{2})\phi_t\text{d}y.
\end{align}
In order to deal with the items containing $\|w_{yy}\|^2 $ and $\|\phi_{y}\|^2 $, multiplying \eqref{1426a} and \eqref{1426b} by $w$ and $\phi$, respectively, yields
\begin{align*}
\int_{0}^{l}w_{tt}w\text{d}y
=&-a_1\|w_{yy}\|^2 -d_{3}(t)w(l,t)
  - \int_{0}^{l}w_{yy}w_{tyy}\text{d}y \\
  &+\int_{0}^{l}f_1(\phi,\phi_t,w_t,d_{1})w\text{d}y,\\
\int_{0}^{l}\phi_{tt}\phi\text{d}y
=& -a_2\|\phi_{y}\|^2+d_{4}(t)\phi(l,t)
  -\int_{0}^{l}\phi_y\phi_{ty}\text{d}y + \int_{0}^{l}f_2(\phi,\phi_t,w_t,d_{2})\phi\text{d}y.
\end{align*}
Note that for any $\eta\in C^2(\mathbb{R}_{\geq 0};L^2(0,l))$, there holds $\frac{d}{dt}\int_{0}^{l}\eta\eta_{t}\text{d}y=\int_{0}^{l}\eta^2_{t}\text{d}y+
\int_{0}^{l}\eta\eta_{tt}\text{d}y$. Then we have
\begin{align}
\frac{d}{dt}\!\int_{0}^{l}\!\!ww_{t}\text{d}y&= -a_1\|w_{yy}\|^2 \!-\!d_{3}(t)w(l,t) - \int_{0}^{l}\!\!w_{yy}w_{tyy}\text{d}y \notag\\
&\ \ \ +\int_{0}^{l}\!\!w^2_{t}\text{d}y+\int_{0}^{l}f_1(\phi,\phi_t,w_t,d_{1})w\text{d}y,\label{+10}\\
\frac{d}{dt}\!\int_{0}^{l}\!\!\phi\phi_{t}\text{d}y&= -a_2\|\phi_{y}\|^2+d_{4}(t)\phi(l,t) -\int_{0}^{l}\phi_y\phi_{ty}\text{d}y \notag\\
&\ \ \  +\int_{0}^{l}\phi^2_{t}\text{d}y
  +\int_{0}^{l}f_2(\phi,\phi_t,w_t,d_{2})\phi\text{d}y.\label{+11}
\end{align}
We define
the augmented energy
\begin{align}\label{+13}
\mathcal {E}(t)=E(t)+\varepsilon_{1}\int_{0}^{l}\phi\phi_{t}\text{d}y+\varepsilon_{2}\int_{0}^{l}ww_{t}\text{d}y,
\end{align}
where $0<\varepsilon_{1}<1$ and $0<\varepsilon_{2}<1$ are constants to be chosen later.

Note that (see \cite{Lhachemi:2017})
\begin{align*}
\left|\int_{0}^{l}ww_{t}\text{d}y\right|
\leq \frac{\max\{1,l^4/2\}}{\sqrt{a_1}}E(t).
\end{align*}
and
\begin{align*}
\left|\int_{0}^{l}\phi\phi_{t}\text{d}y\right|
\leq \frac{\max\{1,l^2/2\}}{\sqrt{a_2}}E(t),
\end{align*}
Choosing $0<\varepsilon_1,\varepsilon_2<\frac{1}{K_m}$, we have
\begin{align}\label{Hugo 25}
\frac{1}{1+K_m\varepsilon_m}\mathcal {E}(t)\leq E(t)\leq \frac{1}{1-K_m\varepsilon_m}\mathcal {E}(t),
\end{align}
where $\varepsilon_m=\max\{\varepsilon_1,\varepsilon_2\}$.

Based on \eqref{+8} to \eqref{+13} and Appendix~\ref{Appendix: b}, we get
\begin{align}
\frac{\text{d}}{\text{d}t}\mathcal {E}(t)
=&\frac{\text{d}}{\text{d}t}E(t)+\varepsilon_{1}\frac{\text{d}}{\text{d}t} \int_{0}^{l}\phi\phi_{t}\text{d}y +\varepsilon_{2}\frac{\text{d}}{\text{d}t}\int_{0}^{l}ww_{t}\text{d}y\notag\\
=& - b_1\|w_{tyy}\|^2-d_{3}(t)w_t(l,t)
   +\int_{0}^{l}f_1(\phi,\phi_t,w_t,d_{1})w_t\text{d}y-b_2\|\phi_{ty}\|^2\notag\\
&+d_{4}(t)\phi_t(l,t)+\int_{0}^{l}\!\!f_2(\phi,\phi_t,w_t,d_{2})\phi_t\text{d}y
 +\varepsilon_{1}\bigg(-a_2\|\phi_{y}\|^2+d_{4}(t)\phi(l,t)\notag\\
& -\int_{0}^{l}\phi_y\phi_{ty}\text{d}y +\int_{0}^{l}\phi^2_{t}\text{d}y + \int_{0}^{l}f_2(\phi,\phi_t,w_t,d_{2})\phi\text{d}y \bigg)
 +\varepsilon_{2}\bigg(\!\!-a_1\|w_{yy}\|^2 \notag\\
&-d_{3}(t)w(l,t) - \int_{0}^{l}w_{yy}w_{tyy}\text{d}y +\int_{0}^{l}w^2_{t}\text{d}y + \int_{0}^{l}f_1(\phi,\phi_t,w_t,d_{1})w\text{d}y\!\!\bigg)\notag\\
=&- b_1\|w_{tyy}\|^2-\varepsilon_2a_1\|w_{yy}\|^2+\varepsilon_2\|w_{t}\|^2 - b_2\|\phi_{ty}\|^2\notag\\
&-\varepsilon_1a_2\|\phi_{y}\|^2+\varepsilon_1\|\phi_{t}\|^2
 -\varepsilon_1 \int_{0}^{l}\phi_y\phi_{ty}\text{d}y-\varepsilon_2\int_{0}^{l}w_{yy}w_{tyy}\text{d}y \notag\\
&+ \int_{0}^{l} f_1(\phi,\phi_t,w_t,d_{1})(w_t+\varepsilon_2w)\text{d}y
 +\int_{0}^{l}f_2(\phi,\phi_t,w_t,d_{2})(\phi_t+\varepsilon_1\phi)\text{d}y\notag\\
&-\big(w_{t}(l,t)+\varepsilon_2w(l,t)\big)d_{3}(t)
 +\big(\phi_{t}(l,t)+\varepsilon_1\phi(l,t)\big)d_{4}(t)\notag\\
\leq& (\varepsilon_2+\Lambda_1)\|w_{t}\|^2+(\Lambda_2-\varepsilon_2a_1)\|w_{yy}\|^2
 +(\varepsilon_1 +\Lambda_3)\|\phi_{t}\|^2\notag\\
&+(\Lambda_4-\varepsilon_1a_2)\|\phi_{y}\|^2+(\Lambda_5-b_2)\|\phi_{ty}\|^2+(\Lambda_6-b_1)\|w_{tyy}\|^2+\Lambda_7\label{1716}\notag\\
\leq & (\varepsilon_2+\Lambda_1)\|w_{t}\|^2+(\Lambda_2-\varepsilon_2a_1)\|w_{yy}\|^2
  +(\varepsilon_1 +\Lambda_3)\|\phi_{t}\|^2 \notag\\
& +(\Lambda_4-\varepsilon_1a_2)\|\phi_{y}\|^2 +\frac{2}{l^2}(\Lambda_5-b_2)\|\phi_{t}\|^2+\frac{4}{l^4}(\Lambda_6-b_1)\|w_{t}\|^2+\Lambda_7\notag\\
\leq &\bigg(\varepsilon_2+\Lambda_1+\frac{4}{l^4}(\Lambda_6-b_1)\bigg)\|w_{t}\|^2
 +(\Lambda_2-\varepsilon_2a_1)\|w_{yy}\|^2\notag\\
&+\bigg(\varepsilon_1 +\Lambda_3+\frac{2}{l^2}(\Lambda_5-b_2)\bigg)\|\phi_{t}\|^2
 +(\Lambda_4-\varepsilon_1a_2)\|\phi_{y}\|^2+\Lambda_7,
\end{align}
with the coefficients satisfying
\begin{subequations}\label{1717}
\begin{align}
&\Lambda_5-b_2<\Lambda_5'-b_2<0,\label{17a}\\
&\Lambda_6-b_1<\Lambda_6'-b_1<0,\label{17b}\\
&\varepsilon_2+\Lambda_1+\frac{4}{l^4}(\Lambda_6-b_1)
<\varepsilon_2+\Lambda_1+\frac{4}{l^4}(\Lambda_6'-b_1)<0,\label{17c}\\
&\Lambda_2-\varepsilon_2a_1<\Lambda_2'-\varepsilon_2a_1<0,\label{17d}\\
&\varepsilon_1
+\Lambda_3+\frac{2}{l^2}(\Lambda_5-b_2)<\varepsilon_1
+\Lambda_3+\frac{2}{l^2}(\Lambda_5'-b_2)<0,\label{17e}\\
&\Lambda_4-\varepsilon_1a_2<\Lambda_4'-\varepsilon_1a_2<0,\label{17f}
\end{align}
\end{subequations}
where $\Lambda_1,\Lambda_2,...,\Lambda_7$ and $\Lambda_2',\Lambda_4',...,\Lambda_7'$ are defined in \eqref{Eq: Lambda def} in Appendix~\ref{Appendix: b}. The proof of the above inequalities are given in Appendix~\ref{Appendix: c}.

Setting  $\mu_{m}=\min\big\{-\varepsilon_2-\Lambda_1-\frac{4}{l^4}(\Lambda_6'-b_1),
-\Lambda_2'+\varepsilon_2a_1,-\varepsilon_1
-\Lambda_3-\frac{2}{l^2}(\Lambda_5'-b_2),-\Lambda_4'+\varepsilon_1a_2\big\}>0$, which is independent of $t$, we obtain from \eqref{Hugo 25} and \eqref{1716}:
\begin{align}
\frac{\text{d}}{\text{d}t}\mathcal {E}(t)
\leq &  -\mu_{m}E(t)+\Lambda_7\notag\\
\leq &-\frac{\mu_m}{1+K_m\varepsilon_m}\mathcal {E}(t)+\Lambda_7\notag\\
\leq &-\frac{\mu_m}{2}\mathcal {E}(t)+\Lambda_7\notag\\
=&-\frac{\mu_m}{2}\mathcal {E}(t)+\frac{\|d_{1}(\cdot,t)\|^2}{2}\bigg(\frac{1}{r_7}+\frac{\varepsilon_2}{r_8} \bigg)
  +\frac{\|d_{2}(\cdot,t)\|^2}{2}\bigg(\!\frac{1}{r_9}+\frac{\varepsilon_1}{r_{10}}\!\bigg)\notag\\
 &
 +2\sqrt{2l}\big(|d_{3}(t)|+|d_{4}(t)\big)\notag\\
\leq & -\frac{\mu_m}{2}\mathcal {E}(t)+C_1\Big(\|d_{1}(\cdot,t)\|^2 +\|d_{2}(\cdot,t)\|^2 + |d_{3}(t)|+|d_{4}(t)|\Big), \label{iISS}\\
\leq &-\frac{\mu_m}{2}\mathcal {E}(t)+ C_1\Big(\|d_{1}\|^2_{L^{\infty}(0,t;L^2(0,l))}+\|d_{2}\|^2_{L^{\infty}(0,t;L^2(0,l))} \notag\\
&+\|d_{3}\|_{L^{\infty}(0,t)}+\|d_{4}\|_{L^{\infty}(0,t)}\Big),\label{BISS}
\end{align}
where $C_1>0$ is a constant independent of $t$.

We infer from {}{Comparison Lemma (see, \cite[Lemma~3.4]{Khalil2001})} and \eqref{BISS} that
\begin{align*}
\mathcal {E}(t)\leq &\mathcal {E}(0)e^{ -\frac{\mu_m}{2}t}+\frac{2C_1}{\mu_m}
\Big(\|d_{1}\|^2_{L^{\infty}(0,t;L^2(0,l))} + \|d_{2}\|^2_{L^{\infty}(0,t;L^2(0,l))} \notag\\
  &+\|d_{3}\|_{L^{\infty}(0,t)}+\|d_{4}\|_{L^{\infty}(0,t)}\Big)(1-e^{ -\frac{\mu_m}{2}t})\notag\\
\leq &\mathcal {E}(0)e^{ -\frac{\mu_m}{2}t} + \frac{2C_1}{\mu_m}\Big( \|d_{1}\|^2_{L^{\infty}(0,t;L^2(0,l))}+\|d_{2}\|^2_{L^{\infty}(0,t;L^2(0,l))} \notag\\
     &+\|d_{3}\|_{L^{\infty}(0,t)}+\|d_{4}\|_{L^{\infty}(0,t)}\Big)\notag\\
\leq &\mathcal {E}(0)e^{ -\frac{\mu_m}{2}t}+ C_2\Big(\|d_{1}\|^2_{L^{\infty}(0,t;L^2(0,l))} +\|d_{2}\|^2_{L^{\infty}(0,t;L^2(0,l))} \notag\\
     &+\|d_{3}\|_{L^{\infty}(0,t)}+\|d_{4}\|_{L^{\infty}(0,t)}\Big),\notag
\end{align*}
where $C_2>0$ is a constant independent of $t$.
%
%
We conclude by \eqref{Hugo 25} and $\varepsilon_m<\frac{1}{K_m}$ that
\begin{align*}
0\leq E(t)\leq &\frac{1}{1-K_m\varepsilon_m}\mathcal {E}(t)\notag\\
\leq & \frac{1}{1-K_m\varepsilon_m}\mathcal {E}(0)e^{ -\frac{\mu_m}{2}t}
      +\frac{C_2}{1-K_m\varepsilon_m}\Big(\|d_{3}\|_{L^{\infty}(0,t)}+\|d_{4}\|_{L^{\infty}(0,t)}\notag\\
     & +\|d_{1}\|^2_{L^{\infty}(0,t;L^2(0,l))}+\|d_{2}\|^2_{L^{\infty}(0,t;L^2(0,l))}\Big)\notag\\
\leq & \frac{1+K_m\varepsilon_m}{1-K_m\varepsilon_m}E(0)e^{ -\frac{\mu_m}{2}t}
       +\frac{C_2}{1-K_m\varepsilon_m}\Big(\|d_{3}\|_{L^{\infty}(0,t)}+\|d_{4}\|_{L^{\infty}(0,t)}\notag\\
     & +\|d_{1}\|^2_{L^{\infty}(0,t;L^2(0,l))}+\|d_{2}\|^2_{L^{\infty}(0,t;L^2(0,l))}\Big)\notag\\
\leq & C_3E(0)e^{ -\frac{\mu_m}{2}t} + C_3\Big(\|d_{3}\|_{L^{\infty}(0,t)}+\|d_{4}\|_{L^{\infty}(0,t)}\notag\\
     & +\|d_{1}\|^2_{L^{\infty}(0,t;L^2(0,l))} +\|d_{2}\|^2_{L^{\infty}(0,t;L^2(0,l))}\Big),\notag
\end{align*}
where $C_3>0$ is a constant independent of $t$. {}{Noting that since $\|X(\cdot,t)\|_{\mathcal{H}}^2=2E(t)$ for all $t\geq 0$, and $(a+b)^{\frac{1}{2}}\leq a^{\frac{1}{2}}+b^{\frac{1}{2}}$ for all $a\geq 0,b\geq 0$, the claimed result \eqref{main result 1} follows immediately.}

Similarly, we get by \eqref{iISS} and {}{Comparison Lemma}
\begin{align*}
\mathcal{E}(t)\leq &\mathcal {E}(0)e^{ -\frac{\mu_m}{2}t} + {}{C_4}\int_{0}^t\big(|d_{3}(s)|+|d_{4}(s)|\big)\text{d}s \notag\\
                   & +{}{C_4}\int_{0}^t\big(\|d_{1}(\cdot,s)\|^2+\|d_{2}(\cdot,s)\|^2\big)\text{d}s.\notag
\end{align*}
where ${}{C_4}>0$ is a constant independent of $t$. Hence, it follows from \eqref{Hugo 25} that
\begin{align*}
E(t)\leq & {}{C_5}E(0)e^{ -\frac{\mu_m}{2}t}+ {}{C_5}\int_{0}^t\big(|d_{3}(s)|+|d_{4}(s)|\big)\text{d}s \notag\\
         & +{}{C_5}\int_{0}^t\big(\|d_{1}(\cdot,s)\|^2+\|d_2(\cdot,s)\|^2\big)\text{d}s.\notag
   \end{align*}
where ${}{C_5>0}$ is a constant independent of $t$. {}{Finally we conclude \eqref{main result 2} as above.}

\end{proof}
{}{Note that 
\begin{align*}
&\|\phi(\cdot,t)\|_{L^{\infty}(0,l)}^2\leq 2l\|\phi_y\|^2\leq \frac{4l}{a_2}E(t),\notag\\
&\|w_y(\cdot,t)\|_{L^{\infty}(0,l)}^2\leq \frac{l^2}{2}\|w_{yy}\|^2\leq \frac{l^2}{a_1}E(t),\\
&\|w(\cdot,t)\|_{L^{\infty}(0,l)}^2\leq 2l\|w_y\|^2\leq l^3\|w_{yy}\|^2 \leq \frac{2l^3}{a_1}E(t).
   \end{align*}
  We have the following boundedness estimates for the solution of System~\eqref{system 1426}.}
\begin{corollary} {}{Under the same assumptions as in Theorem \ref{Th: EISpS}, the following estimates hold true:
\begin{align*}
&\|w(\cdot, t)\|^2_{L^\infty(0,l)}+\|w_y(\cdot, t)\|^2_{L^\infty(0,l)}+\|\phi(\cdot, t)\|^2_{L^\infty(0,l)}\notag\\
\leq & CE(0)e^{-\frac{\mu_m}{2}t} + C\Big(\|d_{1}\|^2_{L^\infty(0,t;L^2(0,l))} +\|d_{2}\|^2_{L^{\infty}(0,t;L^2(0,l))}\Big)\notag\\
     & +C\Big(\|d_{3}\|_{L^{\infty}(0,t)}+\|d_{4}\|_{L^{\infty}(0,t)}\Big),
\end{align*}
and
\begin{align*}
&\|w(\cdot, t)\|^2_{L^\infty(0,l)}+\|w_y(\cdot, t)\|^2_{L^\infty(0,l)}+\|\phi(\cdot, t)\|^2_{L^\infty(0,l)}\notag\\
\leq & CE(0)e^{-\frac{\mu_m}{2}t} + C\int_{0}^t\Big(\|d_{1}(\cdot,s)\|_{L^2(0,l)}^2+\|d_{2}(\cdot,s)\|_{L^2(0,l)}^2\Big)\text{d}s\notag\\
     & +C\int_{0}^t\big(|d_{3}(s)|+|d_{4}(s)|\big)\text{d}s ,
\end{align*}
where $C>0$ and $\mu_m>0$ are some constants independent of $t$.}
\end{corollary}
Note that the boundedness assumption on $d_3$ and $d_4$ can be relaxed and the structural conditions in \eqref{assumption} can be simplified. Indeed, we estimate $I_4$ and $I_5$ in Appendix~\ref{Appendix: b} as follows:
 \begin{align*}
I_4&:=-(w_{t}(l,t)+\varepsilon_2w(l,t))d_{3}(t)\notag\\
&\leq
\frac{1}{2r_{13}}d_3^2(t)+\frac{r_{13}}{2}(w^2_{t}(l,t)+\varepsilon_2^2w^2(l,t))\notag\\
&\leq \frac{1}{2r_{13}}d_3^2(t)+lr_{13}(\|w_{ty}\|^2+\varepsilon_2^2\|w_{y}\|^2)\notag\\
&\leq\frac{1}{2r_{13}}d_3^2(t)+\frac{l^3r_{13}}{2}(\|w_{tyy}\|^2+\varepsilon_2^2\|w_{yy}\|^2),\ \forall r_{13}>0,
   \end{align*}
and
\begin{align*}
I_5&:=(\phi_{t}(l,t)+\varepsilon_1\phi(l,t))d_{4}(t)\notag\\
&\leq \frac{1}{2r_{14}}d_4^2(t)+l^3r_{14}(\|\phi_{ty}\|^2+\varepsilon_1^2\|\phi_{y}\|^2),\ \forall r_{14}>0.
\end{align*}
 Then the parameters $\Lambda_{2},\Lambda_{4},\Lambda_{5},\Lambda_{6},\Lambda_{7}$ in Appendix~\ref{Appendix: b} become (other parameters retain unchanged)
 \begin{align*}
\Lambda_2&=\frac{\varepsilon_2}{2r_{12}}+\lambda_2+\frac{\varepsilon_2^2l^3r_{13}}{2},\\
\Lambda_4&=\frac{\varepsilon_1}{2r_{11}}
+\lambda_4+\lambda_8+lr_{14}\varepsilon_1^2, \\
\Lambda_5&=\frac{\varepsilon_1}{2}r_{11}+lr_{14}, \\
\Lambda_6&=\frac{\varepsilon_2}{2}r_{12}+\frac{l^3r_{13}}{2}, \\
\Lambda_7&=\lambda_5+\lambda_9+\frac{1}{2r_{13}}d^2_3(t)+\frac{1}{2r_{14}}d^2_4(t).
\end{align*}
If we replace the structural conditions \eqref{assumption} by
\begin{subequations}\label{assumption 2}
\begin{align}
&\varepsilon_2+\Lambda_1+\frac{4}{l^4}(\Lambda_6-b_1)
<0,\\
&\Lambda_2-\varepsilon_2a_1<0,\\
&\varepsilon_1
+\Lambda_3+\frac{2}{l^2}(\Lambda_5-b_2)<0,\\
&\Lambda_4-\varepsilon_1a_2<0,
\end{align}
\end{subequations}
for some $r_{1},r_{2},...,r_{14}, \varepsilon_1,\varepsilon_2$, and relax the boundedness of $d_{3}$ and $d_{4}$, then we have:
\begin{theorem}\label{Th: EISS}
Under the structural assumptions given in \eqref{assumption 2} and assuming that $d_{1},d_{2}\in C^1(\mathbb{R}_{\geq 0};L^2(0,l))$ and ${d_{3},d_4}\in C^2(\mathbb{R}_{\geq 0})$, System~\eqref{system 1426} is EISS and EiISS, having the following estimates:
{}{\begin{align*}
\|X(\cdot,t)\|_{\mathcal{H}}
\leq & Ce^{ -\frac{\mu_m}{4}t}\|X_0\|_{\mathcal{H}} + C\Big(\|d_{1}\|_{L^\infty(0,t;L^2(0,l))} +\|d_{2}\|_{L^{\infty}(0,t;L^2(0,l))}\Big)\notag\\
     & +C\Big(\|d_{3}\|_{L^{\infty}(0,t)}+\|d_{4}\|_{L^{\infty}(0,t)}\Big),
\end{align*}
and
\begin{align*}
\|X(\cdot,t)\|_{\mathcal{H}}
\leq & Ce^{ -\frac{\mu_m}{4}t}\|X_0\|_{\mathcal{H}}+ C\bigg(\int_{0}^t\|d_{1}(\cdot,s)\|_{L^2(0,l)}^2 \text{d}s\bigg)^{\frac{1}{2}}\notag\\
&+C\bigg(\int_{0}^t\|d_{2}(\cdot,s)\|_{L^2(0,l)}^2\text{d}s\bigg)^{\frac{1}{2}}
      +C\bigg(\int_{0}^td_{3}^2(s)\text{d}s\bigg)^{\frac{1}{2}}\notag\\
&+C\bigg(\int_{0}^td_{4}^2(s)\text{d}s\bigg)^{\frac{1}{2}},
\end{align*}
where $C>0$ and $\mu_m>0$ are constants independent of $t$.}
\end{theorem}
\begin{remark}
If $d_3(t)=k_1(w_{t}(l,t)+\varepsilon_2w(l,t))
$ and $d_4(t)=-k_2(\phi_{t}(l,t)+\varepsilon_1\phi(l,t))$ appear as the feedback controls with constants $k_1\geq 0$ and $k_2\geq 0$, and $(w,\phi)$ is the solution of System~\eqref{system 1426}, then the following estimates hold:
\begin{align*}
     E(t)\leq  C{E(0)}e^{-\frac{\mu_m}{2}t} +C\Big(\|d_{1}\|^2_{L^\infty(0,t;L^2(0,l))} +\|d_{2}\|^2_{L^{\infty}(0,t;L^2(0,l))}\Big),
\end{align*}
and
\begin{align*}
     E(t)\leq C{E(0)}e^{-\frac{\mu_m}{2}t}+C{\int_{0}^t\Big(\|d_{1}(\cdot,s)\|_{L^2(0,l)}^2
+\|d_{2}(\cdot,s)\|_{L^2(0,l)}^2\Big)\text{d}s},
   \end{align*}
where $C>0$ and $\mu_m>0$ are some constants independent of $t$, $d_{1}$, and $d_{2}$.

Indeed, in this case, $I_4$ and $I_5$ given in \eqref{+15} and \eqref{+16} in Appendix~\ref{Appendix: b} become
$I_4=-(w_{t}(l,t)+\varepsilon_2w(l,t))(a_1w_{yy}+b_1w_{tyy})_y(l,t)=-k_1(w_{t}(l,t)+\varepsilon_2w(l,t))^2\leq 0$ and $I_5=(\phi_{t}(l,t)+\varepsilon_1\phi(l,t))(a_2\phi_y+b_2\phi_{ty})(l,t)=-k_2(\phi_{t}(l,t)+\varepsilon_1\phi(l,t))^2\leq 0$. Then taking in \eqref{Eq: Lambda def} $M_1=M_2=0$ and proceeding as the proof of Theorem~\ref{Th: EISpS}, one may get the desired results.

Note that, under the above assumptions, a disturbance-free setting (i.e. $d_{1}=d_{2}=0$ in \eqref{1426a} and \eqref{1426b}) was considered in \cite{Lhachemi:2017}, and the exponential stability was obtained.
\end{remark}

\begin{remark}
A more generic setting is to replace the boundary conditions given in \eqref{1426c} by $(a_1w_{yy}+b_1w_{tyy})_y(l,t)= d_3(t)+k_1(w_{t}(l,t)+\varepsilon_2w(l,t))$ and $(a_2\phi_y+b_2\phi_{ty})(l,t)$ $= d_4(t)-k_2(\phi_{t}(l,t) + \varepsilon_1\phi(l,t))$, where $d_3(t),d_4(t)$ are disturbances, $k_1\geq 0$ and $k_2\geq 0$. Under the same assumptions on $d_1$, $d_2$, $d_3$, and $d_4$ as in Theorem~\ref{Th: EISpS} or Theorem~\ref{Th: EISS}, if $(w,\phi)$ is the solution of System \eqref{system 1426} with the above boundary conditions, then it can verify that the ISS and iISS properties given in Theorem~\ref{Th: EISpS} or Theorem~\ref{Th: EISS} hold.
\end{remark}
\begin{remark} {}{As pointed out in \cite{Karafyllis:2016a, Zheng:2017}, the assumptions on the continuities of the disturbances are required for assessing the well-posedness of
the considered system. However, they are only sufficient conditions
and can be weakened if solutions in a weak sense are considered.
Moreover, as shown in the proof of Theorem 2, the assumptions
on the continuities of disturbances can eventually be relaxed for the establishment of ISS estimates.}
\end{remark}

{}{}
{
\section{Simulation results}\label{Sec: Simulation}
The ISS properties of System~\eqref{system 1426} are illustrated in this section. Numerical simulations are performed based on the Galerkin method. The numerical values of the parameters are set to $a_1 = 3$, $b_1=0.3$, $c_1=0.06$, $p_1=q_1=0.04$,
$a_2=5$, $b_2=0.5$, $c_2=0.08$, $p_2=q_2=0.06$, and $l=1$. The four perturbation signals are selected as follows:
\begin{align*}
d_1(y,t) &= 2(1+e^{-0.3t})(1 + \sin(0.5\pi t) + 3 \sin(5\pi t) )y, \\
d_2(y,t) &= - 0.2(1+e^{-0.3t})(1 + \sin(0.5\pi t) + 3 \sin(5\pi t) )y, \\
d_3(t) &= (1+2e^{-0.2t}) \cos(0.2 \pi t) \sin(3 \pi t), \\
d_4(t) &= 0.5(1+e^{-0.2t}) \sin(0.2 \pi t) \cos(3 \pi t),
\end{align*}
while the initial conditions are set to $w_0 = 0.15y^2(y-3l)/(6l^2) \,\mathrm{m}$ and $\phi_0(y) = 8 y^2/l^2 \,\mathrm{deg}$. The system response is depicted in Fig.~\ref{fig: sim} for the flexible displacements over the time and spatial domains. The behavior at the tip, exhibiting the displacements with maximal amplitude, are depicted in Fig.~\ref{fig: tip}. It can be seen that the non zero initial condition vanishes due to the exponential stability of the underlying $C_0$-semigroup. Furthermore, the amplitude of the flexible displacements under bounded in-domain and boundary perturbations remains bounded, which confirms the theoretical analysis.

%

\section{Concluding Remarks}\label{Sec: Conclusion}
The present work established the exponential input-to-state stability (EISS) and exponential integral input-to-state stability (EiISS) of a system of boundary controlled partial differential equations (PDEs) with respect to boundary and in-domain disturbances. Compared to the ISS property with respect to in-domain disturbances, the case of boundary disturbances is more challenging due to essentially regularity issues. This difficulty has been overcome by using \emph{a priori} estimates of the solution to the original PDEs, which leads to ISS gains in the expected form. It should be noted that the Lyapunov functional candidate used in this work is greatly inspired by the results reported \cite{Lhachemi:2017}. As a further direction of research, it may be interesting to introduce and develop tools allowing the establishment of the ISS property for a wider range of problems in a more systematic manner, such as the attempt presented in \cite{Zheng:2017}.

\appendix
\section{Proof of \eqref{1716}}\label{Appendix: b}

By Young's inequality (see, e.g., \cite[Appendix B.2.d]{Evans:2010}) and Poincar\'{e} inequality (see, e.g., \cite[Chap.~2, Remark~2.2]{Krstic:2008}), we have
\begin{align}
\int_{0}^{l} \phi ^2\text{d}y
  \leq & \frac{4l^2}{\pi}\int_{0}^{l} \phi_y^2\text{d}y\leq \frac{l^2}{2}\int_{0}^{l} \phi_y^2\text{d}y,\notag\\
\int_{0}^{l} \phi w_t\text{d}y
  \leq & \frac{1}{2r_1}\int_{0}^{l} \phi^2\text{d}y +\frac{r_1}{2}\int_{0}^{l}w_t^2\text{d}y
  \leq \frac{l^2}{4r_1}\int_{0}^{l}\phi_y^2\text{d}y +\frac{r_1}{2}\int_{0}^{l} w_t^2\text{d}y,\notag\\
\int_{0}^{l}\phi_t w_t\text{d}y
  \leq &\frac{1}{2r_2}\int_{0}^{l}\phi_t^2\text{d}y +\frac{r_2}{2}\int_{0}^{l}w_t^2\text{d}y,\notag\\
\int_{0}^{l}\phi w\text{d}y
  \leq &\frac{1}{2r_3}\int_{0}^{l}\phi^2\text{d}y +\frac{r_3}{2}\int_{0}^{l}w^2\text{d}y
  \leq \frac{l^2}{4r_3}\int_{0}^{l}\phi_y^2\text{d}y +\frac{r_3l^4}{8}\int_{0}^{l}w_{yy}^2\text{d}y,\notag\\
\int_{0}^{l}\phi_t w\text{d}y
  \leq &\frac{1}{2r_4}\int_{0}^{l}\phi_t^2\text{d}y +\frac{r_4}{2}\int_{0}^{l}w^2\text{d}y
  \leq \frac{1}{2r_4}\int_{0}^{l}\phi_t^2\text{d}y +\frac{r_4l^4}{8}\int_{0}^{l}w_{yy}^2\text{d}y,\notag\\
\int_{0}^{l}ww_t\text{d}y
  \leq &\frac{1}{2r_5}\int_{0}^{l}w^2\text{d}y +\frac{r_5}{2}\int_{0}^{l}w_t^2\text{d}y
  \leq \frac{l^4}{8r_5}\int_{0}^{l}w_{yy}^2\text{d}y +\frac{r_5}{2}\int_{0}^{l}w_{t}^2\text{d}y,\notag\\
\int_{0}^{l}\phi \phi_t\text{d}y
  \leq &\frac{1}{2r_6}\int_{0}^{l}\phi^2\text{d}y +\frac{r_6}{2}\int_{0}^{l}\phi_t^2\text{d}y
  \leq \frac{l^2}{4r_6}\int_{0}^{l}\phi_y^2\text{d}y +\frac{r_6}{2}\int_{0}^{l}\phi_t^2\text{d}y,\notag\\
\int_{0}^{l}d_{1} w_t\text{d}y
  \leq &\frac{\|d_{1}(\cdot,t)\|^2}{2r_7}+\frac{r_7}{2}\int_{0}^{l}w_t^2\text{d}y,\notag\\
\int_{0}^{l}d_{1} w\text{d}y
  \leq &\frac{\|d_{1}(\cdot,t)\|^2}{2r_8}+\frac{r_8}{2}\int_{0}^{l}w^2\text{d}y
  \leq \frac{\|d_{1}(\cdot,t)\|^2}{2r_8}+\frac{r_8l^4}{8}\int_{0}^{l}w_{yy}^2\text{d}y  ,\notag\\
\int_{0}^{l}d_{2} \phi_t\text{d}y
  \leq &\frac{\|d_{2}(\cdot,t)\|^2}{2r_9}+\frac{r_9}{2}\int_{0}^{l}\phi_t^2\text{d}y,\notag\\
\int_{0}^{l}d_{2} \phi\text{d}y
  \leq &\frac{\|d_{2}(\cdot,t)\|^2}{2r_{10}}+\frac{r_{10}}{2}\int_{0}^{l}\phi^2\text{d}y
  \leq \frac{\|d_{2}(\cdot,t)\|^2}{2r_{10}}+\frac{r_{10}l^2}{4}\int_{0}^{l}\phi_y^2\text{d}y,\notag\\
\int_{0}^{l}\phi_y\phi_{ty}\text{d}y
  \leq &\frac{1}{2r_{11}}\int_{0}^{l}\phi_y^2\text{d}y+\frac{r_{11}}{2}\int_{0}^{l}\phi_{ty}^2\text{d}y,\notag\\
\int_{0}^{l}w_{yy}w_{tyy}\text{d}y
  \leq &\frac{1}{2r_{12}}\int_{0}^{l}w_{yy}^2\text{d}y+\frac{r_{12}}{2}\int_{0}^{l}w_{tyy}^2\text{d}y.\notag
\end{align}
Then we get
\begin{align}
I_1:=&\int_{0}^{l}f_1(\phi,\phi_t,w_t,d_{1})(w_t+\varepsilon_2w)\text{d}y\notag\\
    =&\int_{0}^{l}(c_1\phi+p_{1}\phi_t+q_{1}w_t+d_{1})(w_t+\varepsilon_2w)\text{d}y\notag\\
\leq &\frac{1}{2}\left(c_1r_1+p_{1}r_2+2{}{-q_1}+r_7+\varepsilon_2{}{-q_1}r_5\right)\|w_t\|^2\notag\\
     &+\frac{\varepsilon_2l^4}{8}\bigg(c_1r_3+p_{1}r_4-\frac{{}{q_1}}{r_5}+r_8\bigg)\|w_{yy}\|^2\notag\\
     &+\frac{p_{1}}{2}\bigg(\frac{1}{r_2}+\frac{\varepsilon_2}{r_4}\bigg)\|\phi_t\|^2
      +\frac{c_1l^2}{4}\bigg(\frac{1}{r_1}+\frac{\varepsilon_2}{r_3}\bigg)\|\phi_y\|^2
      +\frac{\|d_{1}(\cdot,t)\|^2}{2}\bigg(\frac{1}{r_7}+\frac{\varepsilon_2}{r_8}\bigg) \notag\\
   :=&\lambda_1 \|w_t\|^2+\lambda_2 \|w_{yy}\|^2+\lambda_3 \|\phi_t\|^2+\lambda_4 \|\phi_y\|^2+\lambda_5,\label{+20}\\
I_2:=&\int_{0}^{l}f_2(\phi,\phi_t,w_t,d_{2})(\phi_t+\varepsilon_1\phi)\text{d}y\notag\\
 =&\int_{0}^{l}(c_2\phi+p_{2}\phi_t+q_{2}w_t+d_{2})(\phi_t+\varepsilon_1\phi)\text{d}y \notag\\
\leq &\frac{q_{2}}{2}\left(r_2+\varepsilon_1r_1\right)\|w_t\|^2+\frac{1}{2}\bigg(c_2r_6-
2{}{p_{2}}+\frac{q_{2}}{r_2}+ r_9 - \varepsilon_1{}{p_{2}}r_6 \bigg)\|\phi_t\|^2\notag\\
&  +\frac{l^2}{4}\bigg(\frac{c_2}{r_6}+2\varepsilon_1c_2
   -\frac{\varepsilon_1{}{p_{2}}}{r_6}+\frac{\varepsilon_1q_{2}}{r_1}+\varepsilon_1r_{10}\bigg)\|\phi_y\|^2
   +\frac{\|d_{2}(\cdot,t)\|^2}{2}\bigg(\frac{1}{r_9}+\frac{\varepsilon_1}{r_{10}}\bigg)\notag\\
:=&\lambda_6 \|w_t\|^2+\lambda_7 \|\phi_t\|^2+\lambda_8 \|\phi_y\|^2+\lambda_9,\label{+21}\\
I_3:=&\varepsilon_1 \int_{0}^{l}\phi_y\phi_{ty}\text{d}y+\varepsilon_2\int_{0}^{l}w_{yy}w_{tyy}\text{d}y\notag\\
 \leq & \frac{\varepsilon_1}{2}\bigg( \frac{1}{r_{11}}\|\phi_y\|^2+r_{11}\|\phi_{ty}\|^2\bigg)
       +\frac{\varepsilon_2}{2}\bigg(\frac{1}{r_{12}}\|w_{yy}\|^2+r_{12}\|w_{tyy}\|^2\bigg).
\end{align}
We shall estimate $(w_{t}(l,t)+\varepsilon_2w(l,t))d_{3}(t)
$ and $(\phi_{t}(l,t)+\varepsilon_1\phi(l,t))d_{4}(t)$. Note that for any $f \in H^{1}([0,l])$ with $f(0)=0$ there holds $ f^2(l)\leq 2l \|f_y\|^2$. We compute
\begin{align}\label{+15}
I_4&:=-(w_{t}(l,t)+\varepsilon_2w(l,t))d_{3}(t)\notag\\
&\leq
|d_{3}(t)||w_{t}(l,t)+\varepsilon_2w(l,t)|\notag\\
&\leq |d_{3}(t)|( \sqrt{2l}\|w_{ty}\|+\varepsilon_2\sqrt{2l}\|w_{y}\|)\notag\\
&\leq \sqrt{2l}|d_{3}(t)|(2+\|w_{ty}\|^2
+\varepsilon_2\|w_{y}\|^2 )\notag\\
&\leq \sqrt{2l}|d_{3}(t)|\bigg(2+\frac{l^2}{2}\|w_{tyy}\|^2
+\frac{\varepsilon_2l^2}{2}\|w_{yy}\|^2 \bigg).
   \end{align}
Similarly, we get
\begin{align}\label{+16}
I_5&:=(\phi_{t}(l,t)+\varepsilon_1\phi(l,t))d_{4}(t)\leq \sqrt{2l}|d_{4}(t)|\big(2+\|\phi_{ty}\|^2
+\varepsilon_1\|\phi_{y}\|^2 \big).
\end{align}
Finally, we have
\begin{align*}
&I_1\!+\!I_2\!+\!I_3\!+\!I_4\!+\!I_5 \notag\\
\leq &\Lambda_1\|w_{t}\|^2+\Lambda_2\|w_{yy}\|^2+
            \Lambda_3\|\phi_{t}\|^2+\Lambda_4\|\phi_{y}\|^2 +\Lambda_5\|\phi_{ty}\|^2+\Lambda_6\|w_{tyy}\|^2\!+\!\Lambda_7,
\end{align*}
where
\begin{subequations}\label{Eq: Lambda def}
\begin{align}
\Lambda_1&=\lambda_1+\lambda_6,\\
\Lambda_2&=\frac{\varepsilon_2}{2r_{12}}+\lambda_2+\frac{\varepsilon_2l^2}{2}\sqrt{2l} |d_3(t)|
        \leq\frac{\varepsilon_2}{2r_{12}}+\lambda_2+\frac{\varepsilon_2l^2}{2}\sqrt{2l} M_1:=\Lambda_2',\\
\Lambda_3&=\lambda_3+\lambda_7, \\
\Lambda_4&=\frac{\varepsilon_1}{2r_{11}}
+\lambda_4+\lambda_8+\varepsilon_1\sqrt{2l}|d_4(t)|
 \leq\frac{\varepsilon_1}{2r_{11}}
+\lambda_4+\lambda_8+\varepsilon_1\sqrt{2l}M_2:=\Lambda_4', \\
\Lambda_5&=\frac{\varepsilon_1}{2}r_{11}+\sqrt{2l}|d_4(t)|
           \leq\frac{\varepsilon_1}{2}r_{11}+\sqrt{2l}M_2:=\Lambda_5', \\
\Lambda_6&=\frac{\varepsilon_2}{2}r_{12}+\frac{l^2}{2}\sqrt{2l}|d_3(t)|
           \leq\frac{\varepsilon_2}{2}r_{12}+\frac{l^2}{2}\sqrt{2l}M_1:=\Lambda_6', \\
\Lambda_7&=\lambda_5+\lambda_9+2\sqrt{2l}(|d_3(t)|+|d_4(t)|)
           \leq \lambda_5+\lambda_9+2\sqrt{2l}(M_1+M_2):=\Lambda_7',\\
M_1&=\|d_3\|_{L^{\infty} \mathbb{R}_{\geq 0}},\\
M_2&=\|d_4\|_{L^{\infty} \mathbb{R}_{\geq 0}}.
\end{align}
\end{subequations}

\section{Proof of \eqref{1717}}\label{Appendix: c}
First, note that
\begin{subequations}\label{1723}
\begin{align}
&\varepsilon_2+\Lambda_1+\frac{4}{l^4}(\Lambda_6'-b_1)<0 \notag\\
\Leftrightarrow & \varepsilon_2 +\frac{1}{2}(c_1r_1+p_{1}r_2-2{}{q_1}+r_7-\varepsilon_2{}{q_1}r_5)\notag\\
&+\frac{q_{2}}{2}(r_2+\varepsilon_1r_1)+\frac{4}{l^4}\bigg(\frac{\varepsilon_2}{2}r_{12}
+\frac{l^2\sqrt{2l}}{2}M_1-b_1\bigg)<0,\label{1723c}\\
&\Lambda_2'-\varepsilon_2a_1<0  \notag\\
\Leftrightarrow & \frac{1}{2r_{12}}+\frac{l^4}{8}\bigg(c_1r_3+p_{1}r_4-\frac{{}{q_1}}{r_5}+ r_8\bigg) +\frac{l^2\sqrt{2l}}{2}M_1-a_1<0,\label{1723d}\\
&\varepsilon_1 +\Lambda_3+\frac{2}{l^2}(\Lambda_5-b_2)<0 \notag\\
\Leftrightarrow & \varepsilon_1 +\frac{1}{2}\bigg(c_2r_6-2{}{p_{2}}+\frac{q_{2}}{r_2}+r_9-\varepsilon_1{}{p_{2}}r_6 \bigg)\notag\\
& +\frac{p_{1}}{2}\bigg(\frac{1}{r_2}+\frac{\varepsilon_2}{r_4}\bigg)
  +\frac{2}{l^2}\bigg(\frac{\varepsilon_1}{2}r_{11}+\sqrt{2l}M_2-b_2\bigg)<0,\label{1723e}\\
& \Lambda_4'-\varepsilon_1a_2\!<0 \notag\\
\Leftrightarrow & \frac{\varepsilon_1}{2r_{11}}+\frac{c_1l^2}{4}\bigg(\!\frac{1}{r_1}
+\frac{\varepsilon_2}{r_3}\!\bigg)+\varepsilon_1\sqrt{2l}M_2-\varepsilon_1a_2 \notag\\
&+\frac{l^2}{4}\bigg(\!\frac{c_2}{r_6}+
2\varepsilon_1c_2-\frac{\varepsilon_1{}{p_{2}}}{r_6}+\frac{\varepsilon_1{}{q_{2}}}{r_1}
+\varepsilon_1r_{10}\!\bigg)  <0 ,\label{1723f}\\
\eqref{17c}&\Rightarrow \eqref{17b}\ \  \text{and}\ \ \eqref{17e}\Rightarrow \eqref{17a}.
\end{align}
\end{subequations}
It suffices to prove the right hand side of \eqref{1723c}-\eqref{1723f}. 

Indeed, we get from \eqref{8b}
\begin{align}
&\frac{c_1+c_2}{2}(l^2K_m+\sqrt{l})+\sqrt{2l}M_2+\frac{l}{2}+\frac{l^2}{2}(c_2{}{-p_{2}+q_{2}})\notag\\
\leq &\frac{c_1+c_2}{2}(K_m+1)(l^2+\sqrt{l})+\sqrt{2l}M_2+\frac{l}{2} +\frac{l^2}{2}(c_2{}{-p_{2}+q_{2}})\notag\\
\leq &(K_m+1)(l^2+\sqrt{l})\bigg(\frac{c_1+c_2}{2}+\sqrt{2}M_2+\frac{1}{2} +\frac{c_2{}{-p_{2}+q_{2}}}{2}\bigg)\notag\\
\leq &\sqrt{l}(1+l\sqrt{l})(K_m+1)(1+c_1+{}{q_{2}}+c_2{}{-p_{2}}+\sqrt{2}M_2)\notag\\
\leq &\sqrt{2l}(1+l\sqrt{l})(K_m+1)(1+c_1+{}{q_{2}}+c_2{}{-p_{2}}+M_2)\notag\\
< & a_2,\notag
\end{align}
which implies
\begin{align}\label{+24}
\frac{c_1+c_2}{2}\sqrt{l}+\sqrt{2l}M_2+\frac{l}{2}+\frac{l^2}{2}(c_2{}{-p_{2}+q_{2}})<a_2,
\end{align}
and
\begin{align}\label{+25}
\frac{c_1+c_2}{2}l^2K_m+\sqrt{2l}M_2+\frac{l}{2}+\frac{l^2}{2}(c_2{}{-p_{2}+q_{2}})<a_2.
\end{align}

Let
\begin{align*}
\varepsilon_0=\dfrac{\dfrac{(c_1+c_2)l^2}{2}}{a_2-\sqrt{2l}M_2-\dfrac{l}{2}-\dfrac{l^2}{2}(c_2{}{-p_{2}+q_{2}})}.
\end{align*}
By \eqref{+24} and \eqref{+25}, we have $\frac{(c_1+c_2)l^2}{2a_2}<\varepsilon_0 <\min\Big\{\frac{1}{K_m},l\sqrt{l}\Big\}$.

We get from \eqref{8c}
\begin{align}
(c_1+p_{1}{}{-4q_1}+q_{2})l^4+q_{2}l^4\varepsilon_0+8l^2\sqrt{2l}M_1<16b_1.\label{+26}
\end{align}
Indeed, we can compute
\begin{align*}
&(c_1+p_{1}{}{-4q_1}+q_{2})l^4+q_{2}l^4\varepsilon_0+8l^2\sqrt{2l}M_1\notag\\
\leq &(c_1+p_{1}{}{-4q_1}+q_{2})l^4+q_{2}l^4l\sqrt{l} +8l^2\sqrt{2l}M_1\notag\\
 =& l^2\sqrt{l}\bigg( l\sqrt{l}(c_1+p_{1}{}{-4q_1}+q_{2} )+l^3q_{2}+8\sqrt{2}M_1\bigg)\notag\\
\leq & l^2\sqrt{l}(1+l^3)( c_1+p_{1}{}{-4q_1}+q_{2} +8\sqrt{2}M_1)\notag\\
\leq & 8l^2\sqrt{l}(1+l^3)( c_1+p_{1}{}{-q_1}+q_{2} +\sqrt{2}M_1)\notag\\
\leq & 8l^2\sqrt{2l}(1+l^3)( c_1+p_{1}{}{-q_1}+q_{2} +M_1)\notag\\
< & 16b_1.
\end{align*}

We get from \eqref{8d}
\begin{align}
&l^2(p_{1}+\frac{c_2}{4}{}{-p_{2}}+q_{2})+l^2\bigg(1-\frac{{}{p_{2}}}{4}
 +\frac{1}{l^3}\bigg)\varepsilon_0+2\sqrt{2l}M_2< 2b_2. \label{+27}
\end{align}
Indeed, we can compute
\begin{align*}
&l^2\left(p_{1}+\frac{c_2}{4}{}{-p_{2}}+q_{2}\right)+l^2
\bigg(1-\frac{{}{p_{2}}}{4}
 +\frac{1}{l^3}\bigg)\varepsilon_0+2\sqrt{2l}M_2\notag\\
\leq & l^2\left(p_{1}+\frac{c_2}{4}{}{-p_{2}}+q_{2}\right)+l^2\bigg(1-\frac{{}{p_{2}}}{4}
 +\frac{1}{l^3}\bigg)l\sqrt{l} +2\sqrt{2l}M_2\notag\\
\leq & l^2\left(p_{1}+\frac{c_2}{4}{}{-p_{2}}+q_{2}\right)
+\sqrt{l}\bigg(\bigg(1-\frac{{}{p_{2}}}{4}\bigg)l^3+1\bigg)+2\sqrt{2l}M_2\notag\\
\leq & l^2\left(p_{1}+\frac{c_2}{4}{}{-p_{2}}+q_{2}\right)
+\sqrt{l}(l^3+1)\bigg(1-\frac{{}{p_{2}}}{4}\bigg)
+2\sqrt{2l}M_2\notag\\
\leq & \sqrt{l}(l^3+1)\bigg( p_{1}+\frac{c_2}{4}{}{-p_{2}}+q_{2}+1-\frac{{}{p_{2}}}{4}+2\sqrt{2}M_2\bigg)\notag\\
\leq & 2\sqrt{2l}(l^3+1)\bigg( 1+p_{1}+c_2{}{-p_{2}}+q_{2}+M_2\bigg)\notag\\
 < & 2b_2.
\end{align*}

Setting $ r_{11}=\frac{1}{l}$, we have
\begin{align}
&\frac{(c_1+c_2)l^2}{2}+\bigg( (c_2{}{-p_{2}+q_2})\frac{l^2}{2}+\frac{1}{2r_{11}}+\sqrt{2l}M_2-a_2\bigg)\varepsilon_0 =0,\label{+28}
\end{align}
with $(c_2{}{-p_{2}+q_2})\frac{l^2}{2}+\frac{1}{2r_{11}}+\sqrt{2l}M_2-a_2<0$ due to \eqref{+24} or \eqref{+25}.

Regarding \eqref{+26}, \eqref{+27} and \eqref{+28}, we may choose $\varepsilon_{1}\in \big(\varepsilon_{0},\min\{\frac{1}{K_m},l\sqrt{l}\}\big)$ such that
\begin{align}
&\frac{1}{4}(c_1\!+p_{1}\!{}{-4q_1}\!+q_{2})l^4\!+\!\frac{1}{4}q_{2}l^4\varepsilon_1+2l^2\sqrt{2l}M_1\!<\!4b_1,\label{+29} \\
&l^2(p_{1}+\frac{c_2}{4}{}{-p_{2}}+q_{2})+l^2\bigg(1-\frac{{}{p_{2}}}{4}+\frac{1}{l^3}\bigg)\varepsilon_1+2\sqrt{2l}M_2 <2b_2, \label{+30}\\
&\frac{(c_1+c_2)l^2}{2}+\bigg( (c_2{}{-p_{2}+q_{2}})\frac{l^2}{2}+\frac{1}{2r_{11}}+\sqrt{2l}M_2-a_2\bigg)\varepsilon_1  <0.\label{+31}
\end{align}

Note that by \eqref{8a}, \eqref{1723d} holds with $r_3,r_4,r_8$ small enough and $r_5,r_{12}$ large enough.

Setting $r_1=r_6=\frac{1}{2}$, we get by \eqref{+31},
\begin{align}
&\frac{\varepsilon_1}{2r_{11}}+\frac{c_1l^2}{4}\frac{1}{r_1} +\frac{l^2}{4}\bigg(\frac{c_2}{r_6}+
2\varepsilon_1c_2-\frac{\varepsilon_1{}{p_{2}}}{r_6}+\frac{\varepsilon_1{}{q_{2}}}{r_1}\bigg) +\varepsilon_1\sqrt{l} M_2-\varepsilon_1a_2<0. \label{+32}
\end{align}
For the above $r_{3}$, one may choose $\varepsilon_2 <\frac{1}{K_m}$ small enough such that $\frac{\varepsilon_2}{r_3}$ small enough. Then by \eqref{+32}, \eqref{1723f} holds with small $r_{10}$ and $\frac{\varepsilon_2}{r_3}$.

Similarly, by \eqref{+30}, \eqref{1723e} holds with $r_2=\frac{1}{2}$ and $r_9, \frac{\varepsilon_2}{r_4}$ small enough. By \eqref{+29}, \eqref{1723c} holds with $r_2=\frac{1}{2}$ and $\varepsilon_2r_{12}, \varepsilon_2r_5,r_7$ small enough.





\end{document}